\documentclass{amsart}
\usepackage{amsfonts}
\usepackage{amsmath}
\usepackage{amssymb}
\usepackage{textcomp}
\usepackage{tikz}
\usepackage{enumerate}
\usetikzlibrary{matrix,arrows,decorations.pathmorphing}
\usetikzlibrary{decorations.markings}
\newtheorem{theorem}{Theorem}[section]

\newtheorem{corollary}[theorem]{Corollary}

\theoremstyle{definition}

\newtheorem{remark}[theorem]{Remark}

\numberwithin{equation}{section}

\newcommand{\C}{{\mathbb C}}
\newcommand{\Z}{{\mathbb Z}}
\newcommand{\F}{{\mathbb F}}
\newcommand{\bigslant}[2]{{\raisebox{.1em}{$#1$}\left/\raisebox{-.1em}{$#2$}\right.}}
\newtheorem{Theorem}{Theorem}[section]
\newtheorem{Theorem--Definition}[Theorem]{Theorem--Definition}
\newtheorem{Corollary}[Theorem]{Corollary}
\newtheorem{Lemma}[Theorem]{Lemma}
\newtheorem{Lemma--Definition}[Theorem]{Lemma--Definition}

\newtheorem{Proposition--Definition}[Theorem]{Proposition--Definition}

\newtheorem{Remark}[Theorem]{Remark}
\newtheorem{Remark--Definition}[Theorem]{Remark--Definition}%

\address{}

\begin{document}
\title{Crossed Products and Coding Theory}
\vspace{.4in}
\author{Yuval Ginosar}

\address{Department of Mathematics, University of Haifa, Haifa 3498838, Israel}
\email{ginosar@math.haifa.ac.il}
\author{Aviram Rochas Moreno}
\email{cyclemen@gmail.com}

\maketitle

Families of codes such as group codes, constacyclic and skew cyclic codes, some of which independently suggested in the literature,
turn out to be special instances of the general family of crossed product codes.
Hamming-metric is a main feature of ambient code spaces which is used to evaluate the efficiency of their various codes.

This note aims at classifying the ambient spaces of crossed products up to Hamming-isometry.
In \S\ref{back} the structure of crossed products and their connection to coding theory is briefly surveyed,
and the regularity property of relative semisimplicity is explained.
In \S\ref{ham} we recall the notion of Hamming-metric, and establish a criterion
for two crossed products of a group $G$ over a base ring $R$ to be isometric in terms of a certain $G$-automorphism action on the second cohomology of $G$ with coefficients in $R^*$
(Corollary \ref{cor}). This classification is demonstrated in \S\ref{ex} by two families of examples, namely crossed products of cyclic groups over finite fields (Theorem \ref{121cyclic}),
and complex twisted group algebras
of elementary abelian groups (Theorem \ref{complete}). We also determine when crossed products belonging to these families are (relative) semisimple, and in particular,
when they admit only trivial codes.

\section{Crossed products over commutative rings}\label{back}
Let $G$ be a finite group and let $R$ be a commutative ring with 1, whose multiplicative group of units is denoted by $R^*$.
A {\it crossed product} $R*G$ is an associative $G$-graded ring over the base ring $R$ which admits an invertible element in each homogeneous component.
In other words, there is a set of units $\{u_g\}_{g\in G}$ satisfying
\begin{equation}\label{GGA}
R*G=\oplus_{g\in G}Ru_g,
\end{equation}  such that
\begin{equation}\label{GGA1}
 Ru_g\cdot Ru_h=Ru_{gh}, \forall g,h\in G.
\end{equation}
Equation \eqref{GGA} says that any element in $R*G$ is written as $\sum_{g\in G}\beta _gu_g$ with uniquely determined coefficients $\{\beta_g\}_{g\in G}\subset R$.
Equation \eqref{GGA1} says that for every $g,h\in G$
\begin{equation}\label{2coc}
f(g,h):=u_gu_hu_{gh}^{-1}\in R^*.
\end{equation}

It is easily verified that $G$ acts on the ring $R$ by the rule
\begin{equation}\label{Rmod}
g(r):=u_gru_g^{-1},\ \ g\in G, r\in R,
\end{equation}
and that \eqref{GGA} describes a decomposition of $R*G$ as a lattice (free module) over its subring $R$ with a {\it $G$-graded basis} $\{u_g\}_{g\in G}$ over $R$.
The associativity of $R*G$ entails the {\it 2-cocycle condition} on the 2-place function $f:G\times G\to R^*$ given in \eqref{2coc}, that is
$f\in Z^2(G,R^*)$, where $R^*$ is a left $G$-module via \eqref{Rmod}.
Another choice of a graded basis $\{u'_g\}_{g\in G}$ does not change the action of $G$ on $R$, but yields a 2-cocycle $f':G\times G\to R^*$ which differs from $f$ by a {\it coboundary}.
This means that
$[f']=[f]\in H^2(G,R^*)$, where $H^2(G,R^*)$ is the second cohomology group of $G$ with coefficients in $R^*$.
Note that for the trivial element $e\in G$, the invertible element $u_e$ can be chosen to be $1$. In this case, the corresponding 2-cocycle $f\in Z^2(G,R^*)$ satisfies
$f(g,e)=f(e,g)=1\in R^*,$ for all $g\in G,$
and is called {\it normalized}.
When we want to stress the $G$-action on $R$
\begin{equation}\label{action}
\eta:G\to\text{Aut}(R)
\end{equation}
and the cocycle $f\in Z^2(G,R^*)$, we denote the corresponding crossed product by $R_{\eta}^f*G$.

Certain families of crossed products draw special attention. When
the $G$-module structure \eqref{Rmod} on the base ring $R$ is trivial, in other words $R$ is central in $R*G$, then the crossed product
$R_{\eta}^f*G$, or just $R^f*G$, is called a {\it twisted group ring.}
On the other hand, when the 2-cocycle $f\in Z^2(G,R^*)$ given in \eqref{2coc} is identically 1, $R_{\eta}^1*G$ is called a {\it skew group ring}.
A skew group ring which is also a twisted group ring is just an (ordinary) group ring.
An important family of crossed products arises when $R$ is a field and \eqref{action} describes a Galois action whose fixed field is $k$. In this case $R_{\eta}^f*G$ is
$k$-central simple and is called a {\it classical} crossed product.

Ideals of crossed products $R*G$ which are lattices over $R$ may be considered as special instances of {\it codes}.
Indeed, twisted group ring codes have already been investigated
in, e.g. \cite{HP,H01}. These generalize both the well studied families of group ring codes \cite{B67,G}, where the module structure and the 2-cocycles are trivial, and of {\it constacyclic} codes,
where the group $G$ is cyclic \cite{J92}. Skew group ring codes have been shown to yield good parameters for cyclic groups \cite{BGU}. The combination of nontrivial actions and nontrivial 2-cocycles
have also been discussed in the cyclic case \cite{BSU,DNS}. Classical crossed product codes were investigated in, e.g. \cite{SRS}.

To end this section, we point out that the coding theory of a crossed product $R*G$ significantly depends on the condition whether all its ideals which are $R$-lattices are projective as $R*G$-modules.
Crossed products satisfying this condition, termed {\it relative semisimple} in \cite{AGO}, are uniquely decomposed as a direct sum of their irreducible codes.
Relative semisimplicity of twisted group rings over the quotient rings $\Z/p^s\Z$ are studied in \cite{AGO}. When the base ring $R$ is semisimple itself,
then relative semisimplicity of $R*G$ is the same as semisimplicity. A complete criterion for semisimplicity of crossed products is given in \cite[Theorem 3.3]{AGdR}.
\section{Isometries under the Hamming metric}\label{ham}
The main goal of this section is to present a necessary and sufficient cohomological condition for isometry of crossed products (see Corollary \ref{cor}).
Recall that an $R$-basis $\mathcal{B}$ of any $R$-lattice $M$
determines a {\it Hamming weight}
$$\begin{array}{rcl}
\mathcal{H}_{\mathcal{B}}:M &\to &\mathbb{N}\\
\sum_{b\in\mathcal{B}}r_bb& \mapsto & |\{b|\ \ r_b\neq 0\}|,
\end{array}$$
which, in turn, furnishes $M$ with a metric space structure. Note that the unit sphere in the metric space $(M,\mathcal{B})$ is the set of nonzero ``monomials" $\{rb\}_{r\in R,b\in \mathcal{B}}$.
The quality of a code, given as a  sublattice of the based $R$-lattice $(M,\mathcal{B})$, is measured by the minimal Hamming distance between its elements.

An {\it isometry} between two based $R$-lattices $(M,\mathcal{B})$ and $(M',\mathcal{B}')$ is an invertible $R$-module morphism $\rho:M\to M'$ that satisfies
\begin{equation}\label{iso}
\mathcal{H}_{\mathcal{B}'}(\rho(m))=\mathcal{H}_{\mathcal{B}}(m)
\end{equation}
for every $m\in M$. Based lattices are called isometric if such an isometry does exist.
A based lattice can be mapped isometrically to itself.
We say that an $R$-module automorphism $\rho:M\to M$ is an isometry of $(M,\mathcal{B})$ if
\begin{equation}\label{autometry}
\mathcal{H}_{\mathcal{B}}(\rho(m))=\mathcal{H}_{\mathcal{B}}(m),\ \  \forall m\in M.
\end{equation}
We stress that the base $\mathcal{B}$ determines the metric in both sides of \eqref{autometry}.
The isometries of a based $R$-lattice $(M,\mathcal{B})$ evidently form a group under the decomposition operation.

Any isometry $\rho:M\to M'$ between $(M,\mathcal{B})$ and $(M',\mathcal{B}')$ maps the unit sphere of one space onto the unit sphere of the other,
hence yields
a (unique) bijection $\rho':\mathcal{B}\to\mathcal{B}'$
such that
\begin{equation}\label{isometry}
\rho(b)=r_b\rho'(b),\ \ \forall b\in \mathcal{B}
\end{equation}
for some $r_b\in  R^*$.
In fact, given a bijection $\rho':\mathcal{B}\to\mathcal{B}'$ and invertible coefficients $r_b\in  R^*$, condition \eqref{isometry}
is also sufficient for a based $R$-module morphism $\rho:(M,\mathcal{B})\to (M',\mathcal{B}')$ to be an isometry.

Clearly, isometric lattices $M$ and $M'$ are of the same rank, say $n$.
Denote the bases elements $\mathcal{B}=\{b_i\}_{i=1}^n$, and $\mathcal{B}'=\{b'_i\}_{i=1}^n.$
Then the bijection $\rho'$ determines a permutation in the symmetric group $\Sigma_n$ on $n$ elements.
By \eqref{isometry}, the group $\Gamma_n$ of isometries of a based $R$-lattice $(M,\mathcal{B})$ of rank $n$ is generated by two subgroups,
namely the above symmetry group $\Sigma_n$, and the group of invertible diagonal isometries given by $n$-tuples $(r_{b_1},\ldots,r_{b_n})\in (R^*)^n$ as in \eqref{isometry}
(with the identity permutation).
More precisely, $\Gamma_n$ is the wreath product
\begin{equation}\label{monomial}
\Gamma_n=R^*\wr\Sigma_n=(R^*)^n\rtimes\Sigma_{n},
\end{equation}
where $\Sigma_n$ acts on $n$-tuples of $R^*$ by permutations. The group $\Gamma_n$ is called the {\it monomial} group of the lattice $R^n$.

An $R$-lattice $M$ can be equipped with an additional ring structure (admitting $R$ as a subring).
In this case, a code in the ambient space $M$ is usually considered as an ideal of $M$ which is also an $R$-sublattice. An isometry
of based $R$-lattice rings is a based $R$-lattices isometry which is also a morphism of rings.

Crossed products $R_{\eta}^f*G$ are rings with Hamming metric, implicitly determined through the graded basis $\{u_g\}_{g\in G}$.
In order to study isometries of these metric spaces we need the following notion.
Let $G_1$ and $G_2$ be two groups acting on a commutative ring $R$ via
\begin{equation}\label{2actions}
\eta_i:G_i\to\text{Aut}(R),\ \  i=1,2.
\end{equation}
A group morphism $\psi:G_1\to G_2$ is {\it $(\eta_1,\eta_2)$-compatible} if
\begin{equation}\label{compat}
\eta_2\circ\psi=\eta_1.
\end{equation}
The actions \eqref{2actions} are {\it compatible} if there exists an $(\eta_1,\eta_2)$-compatible group isomorphism between $G_1$ and $G_2$.
For every $(\eta_1,\eta_2)$-compatible morphism $\psi:G_1\to G_2$ and $f\in Z^2(G_2,R^*)$, define
\begin{eqnarray}\label{defact}\begin{array}{rcl}
{\psi}(f):G_1\times G_1&\to &R^*\\
(g,g')&\mapsto& f({\psi}(g),{\psi}(g')).
\end{array}\end{eqnarray}
Then a direct computation shows that the 2-place function ${\psi}(f)$ is in $Z^2(G_1,R^*)$.
The above settings give a sufficient condition for two crossed products over $R$ to be isometric.
\begin{Lemma}\label{cpiso}
Let $G_1$ and $G_2$ be two groups acting on a commutative ring $R$ via \eqref{2actions}.
Suppose that there exists an $(\eta_1,\eta_2)$-compatible group isomorphism $\psi:G_1\xrightarrow{\cong} G_2$ such that $\psi(f_2)$ is cohomologous to $f_1$.
 Then the crossed products $R_{\eta_1}^{f_1}*G_1$ and $R_{\eta_2}^{f_2}*G_2$ are isometric.
\end{Lemma}
\begin{proof}
The condition that $\psi(f_2)$ and $f_1$ are cohomologous in $Z^2(G_1,R^*)$ says that there is a map $g\mapsto r_g$ from $G_1$ to $R^*$
such that for every $g,h\in G_1$
\begin{equation}\label{e1}
r_g g(r_h)r_{gh}^{-1}\psi(f_2)(g,h)=f_1(g,h).
\end{equation}

Let $\{u_g\}_{g\in G_1}$ and $\{v_{g'}\}_{g'\in G_2}$ be the graded bases of the crossed products $R_{\eta_1}^{f_1}*G_1$ and $R_{\eta_2}^{f_2}*G_2$ respectively.
Define
$$\begin{array}{rcl}
\rho: R_{\eta_1}^{f_1}*G_1 & \to & R_{\eta_2}^{f_2}*G_2\\
\sum_{g\in G}\beta _gu_g & \mapsto & \sum_{g\in G}\beta _gr_gv_{\psi(g)}.
\end{array}$$
Then $\rho$ is clearly an $R$-module morphism.
Moreover, satisfying condition \eqref{isometry}, $\rho$ is an isometry.
It is thus left to show that it is a morphism of rings. By distributivity, it is enough to check that $\rho$ is multiplicative for monomials $\beta_gu_g,\beta_hu_h$,
where $\beta_g,\beta_h\in R$. Indeed,
\begin{eqnarray}\label{array}
\begin{array}{l}
\rho(\beta_gu_g\beta_hu_h)=\rho(\beta_gg(\beta_h)u_gu_h)=\rho(\beta_gg(\beta_h)f_1(g,h)u_{gh})=\\
=\beta_gg(\beta_h)f_1(g,h)\rho(u_{gh})=\beta_gg(\beta_h)f_1(g,h)r_{gh}v_{\psi(gh)}.
\end{array}
\end{eqnarray}
Plugging \eqref{e1} into \eqref{array} we have
\begin{equation}\label{plug}
\rho(\beta_gu_g\beta_hu_h)=\beta_gg(\beta_h)r_g g(r_h)\psi(f_2)(g,h)v_{\psi(gh)}.
\end{equation}
Now, the fact that $\psi$ is $(\eta_1,\eta_2)$-compatible \eqref{compat} says that $g$ and $\psi(g)$ admit the same action on $R$, that is
$g(r)=\psi(g)(r)$ for every $g\in G_1$ and $r\in R$.
So, by \eqref{plug}
$$\begin{array}{l}
\rho(\beta_gu_g\beta_hu_h)= \beta_g\psi(g)(\beta_h)r_g\psi(g)(r_h)\psi(f_2)(g,h)v_{\psi(gh)}=\\
=\beta_gr_g\psi(g)(\beta_hr_h)f_2(\psi(g),\psi(h))v_{\psi(g)\psi(h)}=\\
 =\beta_gr_g\psi(g)(\beta_hr_h)v_{\psi(g)}v_{\psi(h)}=\beta_gr_gv_{\psi(g)}\beta_hr_hv_{\psi(h)}=\rho(\beta_gu_g)\rho(\beta_hu_h).
\end{array}$$
Thus, $\rho$ is an isometry of crossed products.
\end{proof}
The above condition for the existence of an isometry between two crossed products turns out also to be necessary.
\begin{Lemma}\label{isocp}
Let $G_1$ and $G_2$ be two groups acting on a commutative ring $R$ via \eqref{2actions}.
Then any isometry $\rho:R_{\eta_1}^{f_1}*G_1\to R_{\eta_2}^{f_2}*G_2$ of crossed products
determines a unique $(\eta_1,\eta_2)$-compatible group isomorphism $\psi:G_1\xrightarrow{\cong} G_2$ such that $\psi(f_2)$ is cohomologous to $f_1$.
\end{Lemma}
\begin{proof}
By \eqref{isometry}, $\rho$ gives rise to a bijection $\rho':\{u_g\}_{g\in G_1}\to\{v_{g'}\}_{g'\in G_2}$ of the corresponding graded bases of the two crossed products,
as well as invertible elements $\{r_g\}_{g\in G_1}$ such that $\rho(u_g)=r_g\rho'(u_g)$ for every $g\in G_1$. Certainly, such a bijection $\rho'$ yields
a bijection of group elements $\psi:G_1\to G_2$ such that $\rho'(u_g)=v_{\psi(g)}$ for every $g\in G_1$. Thus,
$$\rho(u_g)=r_g v_{\psi(g)}, \ \ \forall g\in G_1.$$
We show that $\psi$ is a group isomorphism.
For every $g,h\in G_1$
\begin{eqnarray}\label{eq1}
\begin{array}{l}
\rho(u_gu_h)=\rho(u_g)\rho(u_h)=r_g v_{\psi(g)}r_h v_{\psi(h)}=r_g \psi(g)(r_h)v_{\psi(g)} v_{\psi(h)}=\\
r_g \psi(g)(r_h)f_2(\psi(g),\psi(h))v_{\psi(g)\psi(h)}.
\end{array}
\end{eqnarray}
On the other hand
\begin{equation}\label{eq2}
\rho(u_gu_h)=\rho(f_1(g,h)u_{gh})=f_1(g,h)\rho(u_{gh})=f_1(g,h)r_{gh} v_{\psi(gh)}.
\end{equation}
Comparing base elements and coefficients in equations \eqref{eq1} and \eqref{eq2}, we deduce that for every $g,h\in G_1$
\begin{equation}\label{eq3}
v_{\psi(g)\psi(h)}=v_{\psi(gh)},
\end{equation}
as well as
\begin{equation}\label{eq4}
r_g \psi(g)(r_h)f_2(\psi(g),\psi(h))=f_1(g,h)r_{gh}.
\end{equation}
Equation \eqref{eq3} implies that $\psi(g)\psi(h)=\psi(gh)$ for every $g,h\in G_1$, and so $\psi:G_1\to G_2$ is a group isomorphism.

We next show that the isomorphism $\psi$ is $(\eta_1,\eta_2)$-compatible.
The isometry $\rho$ is also an $R$-lattice morphism as well as a ring morphism, so for every $g\in G_1$ and $r\in R$
\begin{equation}\label{equ1}
\rho(u_gr)=\rho(g(r)u_g)=g(r)r_g v_{\psi(g)},
\end{equation}
and on the other hand
\begin{equation}\label{equ2}
\rho(u_gr)=\rho(u_g)\rho(r)=r_g v_{\psi(g)}r=r_g \psi(g)(r)v_{\psi(g)}.
\end{equation}
Equations \eqref{equ1} and \eqref{equ2} yield
\begin{equation}\label{equ3}
g(r)=\psi(g)(r),\ \  \forall g\in G_1, r\in R,
\end{equation}
which says that $\psi$ is $(\eta_1,\eta_2)$-compatible.

Next, from \eqref{eq4} we obtain
\begin{equation}\label{eq5}
f_3(g,h)\psi(f_2)(g,h)=f_1(g,h),
\end{equation}
where $f_3(g,h):=r_g \psi(g)(r_h)r_{gh}^{-1}$. Applying \eqref{equ3} again we have $$f_3(g,h)=r_g g(r_h)r_{gh}^{-1},$$ and hence $f_3:G\times G\to R^*$ is a coboundary (see \eqref{e1}).
By \eqref{eq5}, the cocycles $f_1$ and $\psi(f_2)$ differ by a coboundary $f_3$ and as such they are cohomologous.
\end{proof}
In view of Lemmas \ref{cpiso} and \ref{isocp}, one may get hold of the based $R$-lattice rings which are isometric to a given crossed product $R_{\eta}^f*G$ by restricting only to crossed products
$R_{\eta}^{f'}*G$ of the same group $G$, the same action $\eta$ on $R$ and appropriate cocycles.
To keep the action \eqref{action} in mind, we use the notation $Z^2_{\eta}(G,R^*)$ and $H^2_{\eta}(G,R^*)$.

Let $R$ be a $G$-module via \eqref{action}. Then it is not hard to verify that the set
\begin{equation}\label{auteta}
\text{Aut}_{\eta}(G):=\{\psi\in \text{Aut}(G)|\eta\circ\psi=\eta\}
\end{equation}
of $(\eta,\eta)$-compatible (or just {\it $\eta$-compatible}) automorphisms of $G$ is a subgroup of the group Aut$(G)$, which, in turn, is a subgroup of the symmetric group $\Sigma_{|G|}$.
Note that for a trivial action $\eta$, Aut$_{\eta}(G)=\text{Aut}(G)$.

\begin{Lemma}
With the notation \eqref{defact} and \eqref{auteta},
$$\begin{array}{ccc}
\text{Aut}_{\eta}(G)\times H^2_{\eta}(G,R^*)&\to &H^2_{\eta}(G,R^*)\\
(\psi,[f])&\mapsto& [{\psi}(f)]
\end{array}$$
determines a well-defined (right) action of Aut$_{\eta}(G)$ on $H^2_{\eta}(G,R^*)$.
\end{Lemma}
\begin{proof}
Verify that coboundaries are mapped to coboundaries under automorphisms in Aut$_{\eta}(G)$.
\end{proof}
With this notation, Lemmas \ref{cpiso} and \ref{isocp} yield
\begin{Corollary}\label{cor}
Two crossed products $R_{\eta}^{f}*G$ and $R_{\eta}^{f'}*G$ are isometric if and only if $[f]$ and $[f']$ belong to the same orbit under the Aut$_{\eta}(G)$-action
on $H^2_{\eta}(G,R^*)$. Consequently, fixing an action \eqref{action}, the Hamming isometry classes of crossed products $R_{\eta}^{f}*G$ are in one-to-one correspondence with the quotient set
$H^2_{\eta}(G,R^*)/$Aut$_{\eta}(G)$.
\end{Corollary}
The group of isometries of a crossed product $R_{\eta}^f*G$ is a subgroup of the monomial group
(that is, the isometries only as based $R$-lattices) $\Gamma_{|G|}=(R^*)^{|G|}\rtimes\Sigma_{|G|}$ \eqref{monomial} generated by $(R^*)^{|G|}$ and by a subgroup of $\Sigma_{|G|}$ as follows.
The group of diagonal isometries $(R^*)^{|G|}$ does not change the cohomology class, i.e.
yields crossed products $R_{\eta}^{f'}*G$ such that the cocycles $f'\in Z^2_{\eta}(G,R^*)$ are cohomologous to $f$.
The permutations in $\Sigma_{|G|}$ which take care of the multiplicative property of
the isometries correspond to the compatible automorphisms Aut$_{\eta}(G)$. We have
\begin{Corollary}
The isometry group of a crossed product $R_{\eta}^f*G$ is $(R^*)^{|G|}\rtimes$Aut$_{\eta}(G)$.
\end{Corollary}
\begin{Remark}\label{isomorphic}
If $[f],[f']\in H^2_{\eta}(G,R^*)$ lie in the same orbit under the Aut$_{\eta}(G)$-action as above,
then the corresponding isometric crossed products $R_{\eta}^{f}*G$ and $R_{\eta}^{f'}*G$ are evidently isomorphic as $R$-lattice rings.
The converse, however, does not always hold. That is, an $R$-lattice ring isomorphism between such crossed products does not suffice for $[f],[f']$ to belong to the same Aut$_{\eta}(G)$-orbit.
This is demonstrated in \cite[\S 3.5]{gs}, \cite[Example 3.7]{gordienko} by few examples of isomorphic, but non-isometric twisted group algebras of the same groups.
\end{Remark}

\section{Examples}\label{ex}
This section describes the Hamming isometry classes for two families of crossed products, of cyclic groups over finite fields and of elementary abelian groups
over $\C$.
\subsection{Crossed products of cyclic groups over finite fields}\label{cycodes}
As mentioned above, ideals of crossed products of cyclic groups are the well-investigated skew constacyclic codes \cite{BSU,DNS}.
The isometry classes for the constacyclic case (i.e. for trivial action) were established in \cite{CFLL}.
Earlier, H.Q. Dinh \cite{D12,D13} presented an arithmetic relation between the order of the cyclic group and the cardinality of the base field providing the existence of
constacyclic codes which are not cyclic codes. However,
his dichotomy is finer than isometry, in fact he counts the cohomology classes rather than the Aut$(G)$-orbits of these classes.
The isometry classification of crossed products of cyclic groups is given in Theorem \ref{121cyclic} herein.
The classification of constacyclic code spaces, which is easily derived from this theorem (see Corollary \ref{corol1}), reformulates the result in \cite[Theorem 3.2]{CFLL}.

The ingredients here are a cyclic group $C_n=\langle \sigma\rangle$, a finite field $\F_{q^r}$, where $q$ is any prime number, and an action
\begin{equation}\label{Frobact}
\eta:\begin{array}{rcl}
C_n & \to & \text{Aut}(\F_{q^r})\\
\sigma & \mapsto & \varphi^k,
\end{array}\end{equation}
where $\varphi:x\mapsto x^q, \forall x\in\F_{q^r}$ is the Frobenius automorphism which generates the cyclic group $\text{Aut}(\F_{q^r})\cong C_r$.
So, an isometry classification (see Theorem \ref{121cyclic}) as well as a criterion for (relative) semisimplicity (see Corollary \ref{cyclicss})
should be given in terms of the four integer parameters $n,q,r,k$.
The only constraint on the action \eqref{Frobact} is that the order of $\eta(\sigma)=\varphi^k\in\text{Aut}(\F_{q^r})$ has to divide $n$, which is the order of $\sigma$.
It is not hard to check that $\F_{q^r}$-automorphisms of the same order
(dividing $n$) give rise to compatible actions of $C_n$ on $\F_{q^r}$. Thus, using the notation of \eqref{Frobact} we may assume that
\begin{equation}\label{Delta}
k\in\Delta(r),\text{ and } \frac{r}{k}\in\Delta(n),
\end{equation}
where $\Delta(m)$ denotes the set of positive divisors of $m$.

Before dealing with the isometry problem, it is important to know when $\F_{q^r}*C_n$ is (relative) semisimple, i.e. when it is a direct sum of its irreducible codes (see \S\ref{back})?
By \cite[Theorem 3.3]{AGdR}, a necessary and sufficient condition for the semisimplicity of $\F_{q^r}*C_n$ is that the kernel of the action \eqref{Frobact} is a $q'$-group.
With the assumption \eqref{Delta}, $\ker(\eta)=\langle \sigma^{\frac{r}{k}}\rangle$ is of order $\frac{nk}{r}$. Denoting the greatest common divisor of two integers $m_1,m_2$ by
gcd$(m_1,m_2)$, we have
\begin{corollary}\label{cyclicss}
With the above notation, $\F_{q^r}*C_n$ is semisimple if and only if $\mathrm{gcd}(\frac{nk}{r},q)=1$.
\end{corollary}
\begin{remark}\label{classic}
A stronger condition than the one in Corollary \ref{cyclicss}, namely $nk=r$, says that the action $\eta$ is faithful ($\ker(\eta)$ is trivial). This is exactly the case where
$\F_{q^r}*C_n$ is a classical crossed product. This central simple ambient space admits only trivial codes.
\end{remark}
Back to the Hamming isometry:
owing to Corollary \ref{cor}, we are after the Aut$_{\eta}(C_n)$-orbits in the cohomology of $C_n$ with coefficients in $\F_{q^r}^*$.
Readers who are not interested in the following details are referred directly to equations \eqref{H2eta} and \eqref{autetacn} for the structures of
$H^2_{\eta}(C_n,\F_{q^r}^*)$ and $\mathrm{Aut}_{\eta}(C_n)$ respectively,
and to \eqref{betapsif} for a description of the Aut$_{\eta}(C_n)$-action on $H^2_{\eta}(C_n,\F_{q^r}^*)$. These are exploited for the classification
criterion in Theorem \ref{121cyclic}.

We start with a brief reminder of the structure of $H^2_{\eta}(C_n,\F^*)$ where $\eta$ is any action of $C_n=\langle \sigma\rangle$ on an arbitrary field $\F$.
for more details see, e.g., \cite[Page 58, Example 2]{b}.
Let $f\in Z^2_{\eta}(C_n,\F^*)$, and let $\{u_{\sigma^i}\}_{i=0}^{n-1}$ be a $C_n$-graded basis of the corresponding crossed product $\F_{\eta}^f*C_n$.
Then by a straightforward induction argument we infer that $u_{\sigma}^k=\prod_{i=1}^{k-1}f(\sigma^i,\sigma)u_{\sigma^k}$ for every
positive integer $k$. In particular, $u_{\sigma}^n=\beta_f$, where
\begin{equation}\label{beta}
\beta_f:=\prod_{i=0}^{n-1}f(\sigma^i,\sigma)\in\F^*.
\end{equation}
Since $\beta_f$ is a power of $u_{\sigma}$, it clearly belongs to the $C_n$-invariant elements of $\F^*$, which is denoted by $(\F^*)^{C_n}$.

Next, substituting the basis $\{u_{\sigma^i}\}_{i=0}^{n-1}$ by a new $C_n$-graded basis $\{v_{\sigma^i}\}_{i=0}^{n-1}$ of $\F_{\eta}^f*C_n$
defined by $v_{\sigma^i}:=u_{\sigma}^i,\ \  0\leq i\leq n-1$, we deduce that $f$ is cohomologous to a cocycle

\begin{eqnarray}
\label{standard}(\sigma^l,\sigma^j)\mapsto\left\{\begin{array}{cc}
1,&\text{if } l+j<n\\
\beta_f,& \text{if } l+j\geq n
\end{array}\right., 0\leq l,j\leq n-1.\end{eqnarray}
Consequently, there is an isomorphism
$$\begin{array}{rcl}
\F_{\eta}^f*C_n & \xrightarrow{\cong} & \bigslant{\F[y; \eta]}{\langle y^n-\beta_f \rangle}\\
u_{\sigma}&\mapsto & y+\langle y^n-\beta_f \rangle,
\end{array}$$
where $\F[y; \eta]$ is the skew polynomial ring \cite[\S 1.2]{MR}, whose indeterminate $y$ acts
on the coefficient field $\F$ via the automorphism $\eta(\sigma)$.

Unequal 2-cocycles of the form \eqref{standard} may still be cohomologous.
To complete the discussion, let
\begin{eqnarray}\label{norm}\mathcal{N}_{C_n}:\begin{array}{ccl}
\F^*&\to&\F^*\\
x&\mapsto &\prod_{i=0}^{n-1}\sigma^i(x).
\end{array}\end{eqnarray}
be the {\it norm} map. Then $[f_1]=[f_2]\in H^2_{\eta}(C_n,\F^*)$ if and only if $\beta_{f_1}^{-1}\cdot\beta_{f_2}$ lies in the image of the norm map.
It is easily verified that
$\mathcal{N}_{C_n}(\F^*)< (\F^*)^{C_n},$
and so a cohomology class is identified with an invariant scalar modulo the norm subgroup, that is
\begin{equation}\label{cohomCnF}
H^2_{\eta}(C_n,\F^*)\cong \bigslant{(\F^*)^{C_n}}{\mathcal{N}_{C_n}(\F^*)}.
\end{equation}

We now concentrate on the action \eqref{Frobact} of $C_n$ on the finite field $\F=\F_{q^r}$.
The elements of $\F_{q^r}$ which are fixed under the image $\langle\varphi^k\rangle<\text{Aut}(\F_{q^r})$ of $\eta$ are exactly the elements of the subfield $\F_{q^k}\subset \F_{q^r}$.
Let $x_0$ be any generator of the cyclic group $\F_{q^r}^*$ of order $q^r-1$.
Then \begin{equation}\label{Cninv}
(\F_{q^r}^*)^{C_n}=\F_{q^k}^*=\langle x_0^{\frac{q^r-1}{q^k-1}}\rangle.
\end{equation}
The subgroup $\mathcal{N}_{C_n}(\F_{q^r}^*)< \F_{q^k}$ is generated by $\mathcal{N}_{C_n}(x_0)$.
By the definition of the norm \eqref{norm}, bearing in mind that the subgroup $\langle\sigma^{\frac{r}{k}}\rangle<C_n$ acts trivially on $\F_{q^r}$, we have
\begin{eqnarray}\label{normsbgrp}\begin{array}{l}
\mathcal{N}_{C_n}(x_0)=\prod_{i=0}^{n-1}\sigma^i(x_0)=(\prod_{i=0}^{\frac{r}{k}-1}\sigma^i(x_0))^{\frac{nk}{r}}=(\prod_{i=0}^{\frac{r}{k}-1}\varphi^{ik}(x_0))^{\frac{nk}{r}}\\=
(\prod_{i=0}^{\frac{r}{k}-1}x_0^{q^{ik}})^{\frac{nk}{r}}=(x_0^{\sum_{i=0}^{\frac{r}{k}-1}{q^{ik}}})^{\frac{nk}{r}}=(x_0^{\frac{q^r-1}{q^k-1}})^{\frac{nk}{r}}.
\end{array}\end{eqnarray}
Let $\tilde{x}:=x_0^{\frac{q^r-1}{q^k-1}}$ be a generator of the $C_n$-invariant subgroup of $\F_{q^r}^*$ \eqref{Cninv} of order $q^k-1$.
Gathering equations \eqref{cohomCnF},\eqref{Cninv} and \eqref{normsbgrp} we get
\begin{equation}\label{H2eta}
H^2_{\eta}(C_n,\F_{q^r}^*)\cong \bigslant{\langle\tilde{x}\rangle}{\langle\tilde{x}^{\frac{nk}{r}}\rangle}\cong C_{\text{gcd}(q^k-1,\frac{nk}{r})}.
\end{equation}
Next, the automorphisms of $C_n$ are given by
$$\psi_j:\begin{array}{rcl}C_n&\to& C_n\\
\sigma^i&\mapsto &\sigma^{ij}\end{array},$$
for any $j$ prime to $n$ (that can be chosen between 1 and $n-1$). Which of them is $\eta$-compatible? By the definition in \eqref{auteta},
\begin{eqnarray}\label{autetacn}
\begin{array}{c}\mathrm{Aut}_{\eta}(C_n)=\{\psi_j\in\mathrm{Aut}(C_n) |\ \ \eta=\eta\circ\psi_j\}=\\
=\{\psi_j\in\mathrm{Aut}(C_n) |\ \ \eta(\sigma)=\eta\circ\psi_j(\sigma)\}=\\
\{\psi_j\in\mathrm{Aut}(C_n) |\ \ \varphi^k=\eta(\sigma^j)=\varphi^{kj}\}=\{\psi_j\in\mathrm{Aut}(C_n) |\ \ j\equiv 1 (\text{mod }\frac{r}{k})\}.
\end{array}\end{eqnarray}
Under the identification of $\mathrm{Aut}(C_n)$ with the multiplicative group
$$U_n=\{a(\text{mod }n)\in\Z/n\Z|\ \  \mathrm{gcd}(a,n)=1\},$$
the subgroup $\mathrm{Aut}_{\eta}(C_n)<\mathrm{Aut}(C_n)$ is identified with the kernel of the projection
$$\begin{array}{rcl}
U_n&\to&U_{\frac{r}{k}}\\
a(\text{mod }n)&\mapsto& a (\text{mod }\frac{r}{k}).
\end{array}$$
Its order is therefore $|\mathrm{Aut}_{\eta}(C_n)|=\frac{\phi(n)}{\phi(\frac{r}{k})},$ where $\phi$ denotes Euler's totient function.

Let us now find out what an automorphism $\psi_j$ does to a cohomology class $[f]\in H^2_{\eta}(C_n, \F_{q^r}^*)$.
As explained above, in order to understand the cohomology class
$\psi_j([f])\in H^2(C_n, \F_{q^r}^*)$, it is sufficient to compute $\beta_{\psi_j(f)}$, for a cocycle $f\in Z^2(C_n, \F_{q^r}^*)$ of the form \eqref{standard}.
By the definitions \eqref{defact} and \eqref{beta}
$$\beta_{\psi_j(f)}=\prod_{i=0}^{n-1}\psi_j(f)(\sigma^i,\sigma)=\prod_{i=0}^{n-1}f(\psi_j(\sigma^i),\psi_j(\sigma))=\prod_{i=0}^{n-1}f(\sigma^{ij},\sigma^j).$$
Since $j$ and $n$ are coprime, $\{\sigma^{ij}\}_{i=0}^{n-1}$ runs over all elements of $C_n$ exactly once. Thus
$$\beta_{\psi_j(f)}=\prod_{l=0}^{n-1}f(\sigma^{l},\sigma^j)=\prod_{l=0}^{n-j-1}f(\sigma^{l},\sigma^j)\cdot\prod_{l=n-j}^{n-1}f(\sigma^{l},\sigma^j).$$
Since $f$ is of the form \eqref{standard} we obtain
\begin{equation}\label{betapsif}
\beta_{\psi_j(f)}=1\cdot\prod_{l=n-j}^{n-1}\beta_f=(\beta_f)^j.
\end{equation}
We deduce that the automorphism $\psi_j$ raises cohomology classes (written in a multiplicative way) to the power $j$.

To summarize the discussion we define an equivalence relation $\sim_{\eta}$ on the set $A_{\eta}:=\{1,\cdots,m\}$, where $m:=\text{gcd}(q^k-1,\frac{nk}{r})$.
Two elements $a,b\in A_{\eta}$ are $\sim_{\eta}$-equivalent if $aj\equiv b(\text{mod }m)$ for some integer $j$, which is prime to $n$ and satisfies $j\equiv 1(\text{mod }\frac{r}{k}).$
Applying Corollary \ref{cor}, equations \eqref{H2eta}, \eqref{autetacn} and \eqref{betapsif} together with the above equivalence relation $\sim_{\eta}$ we have
\begin{theorem}\label{121cyclic}
Let \eqref{Frobact} be an action of a cyclic group $C_n$ on a finite field $\F_{q^r}$.
Then there is a one-to-one correspondence between the Hamming isometry classes of the crossed products
$(\F_{q^r})_{\eta}^f*C_n$, and the quotient set $A_{\eta}/\sim_{\eta}$ as above.
\end{theorem}
There are two ``extremal" cases in Theorem \ref{121cyclic}. The first one is when the action \eqref{Frobact} is trivial, that is when $k=r$. In this case, $a\sim_{\eta}b$ if and only if
they differ by $j$ which is prime to $m=\text{gcd}(q^r-1,n)$, and so the equivalence classes can be represented by the set of divisors of $m$. We have
\begin{corollary}\label{corol1}(see \cite[Theorem 3.2]{CFLL})
The twisted hamming isometry classes $(\F_{q^r})^f*C_n$ are in one-to-one correspondence with the set of divisors $\Delta(\text{gcd}(q^r-1,n))$.
\end{corollary}
The other case occurs when the action \eqref{Frobact} is far from being trivial, here in the sense that $m$ divides $\frac{r}{k}$. When this condition holds, $j\equiv 1(\text{mod }\frac{r}{k})$
implies that  $j\equiv 1(\text{mod }m)$, and so $a\sim_{\eta}b$ if and only if $a\equiv b(\text{mod }m).$
\begin{corollary}\label{corol2}
With the above notation, suppose that $m\in\Delta(\frac{r}{k})$. Then there is a one-to-one correspondence between the Hamming isometry classes of the crossed products
$(\F_{q^r})_{\eta}^f*C_n$, and the set $\{1,\cdots,m\}$.
\end{corollary}
When $m$ is fixed, the equivalence relation $\sim_{\eta}$ on $A_{\eta}=\{1,\cdots,m\}$ clearly coarsens (perhaps improperly) the relation under the assumption of Corollary \ref{corol2}.
It is not hard to verify that it always refines (again, perhaps improperly) the relation
for twisted classes in Corollary \ref{corol1}.

\subsection{Twisted group algebras of elementary abelian groups over the complex numbers}
For a prime $p$ and a positive integer $s$, let
$$(C_p)^s=\langle \sigma_1\rangle\times\cdots\times\langle \sigma_s\rangle$$
be the direct product of $s$ copies of the cyclic group of order $p$.
The base ring here is the complex field $\C$ endowed with a trivial $(C_p)^s$-action.
In Theorem \ref{complete} below we classify the complex twisted group algebras $\C^f*(C_p)^s$ up to Hamming isometry.
These are the ambient spaces of complex {\it twisted elementary abelian codes}.

A 2-cocycle $f\in Z^2((C_p)^s,\C^*)$ gives rise to a 2-place function (see, e.g. \cite[\S 2.2]{EK13})
\begin{eqnarray}\label{ssf}\alpha_f:\begin{array}{ccl}
(C_p)^s\times(C_p)^s&\to& \C^*\\
(g_1,g_2)&\mapsto& f(g_1,g_2)f(g_2,g_1)^{-1}.
\end{array}\end{eqnarray}
Then $\alpha_f$ satisfies the following properties:
\begin{enumerate}
\item It is a morphism when fixing each one of the components. In particular, its values lie in the subgroup
$$H_p:=\left\{\exp\left(\frac{2\pi i\cdot j}{p}\right)\right\}_{j=0}^{p-1}<\C^*$$ of $p$-th roots of 1.
\item It is skew-symmetric in the sense that
\begin{equation}\label{skewsymm}\alpha_f(g,g)=1, \ \ \forall g\in (C_p)^s. \end{equation}
Note that equation \eqref{skewsymm} implies that $\alpha_f(g_1,g_2)=\alpha_f(g_2,g_1)^{-1}$ for every $g_1,g_2\in (C_p)^s$.
\end{enumerate}
The function $\alpha_f$ is an {\it alternating bicharacter}.

As already observed by I. Schur \cite{Schur04, Schur07}, two elements $f_1,f_2\in Z^2((C_p)^s,\C^*)$ are cohomologous if and only if they yield the same bicharacters.
We therefore identify the cohomology $H^2((C_p)^s,\C^*)$ with the complex alternating bicharacters on $(C_p)^s$.

How does Aut$((C_p)^s)$ act on alternating bicharacters?
Note that both groups $(C_p)^s$ and $H_p$ can be viewed as vector spaces over the field $\F_p$.
With this perspective, an alternating bicharacter is a skew-symmetric $\F_p$-bilinear form.
Furthermore, the group of automorphisms of $(C_p)^s$ is the general linear group GL$_s(\F_p)$. It acts on the vector space $(C_p)^s$ as a change of basis.
It is now transparent that the action of Aut$((C_p)^s)$ on $H^2((C_p)^s,\C^*)$, i.e. the alternating bicharacters, is by congruence of forms.

Let $\{\sigma_1,\cdots,\sigma_s\}$ be any set of generators of $(C_p)^s$ and let $\zeta_p:=\exp(\frac{2\pi i}{p})\in\C^*$. For every $i=0,\cdots,\lfloor\frac{s}{2}\rfloor$,
define a bicharacter $\alpha_i$ by evaluating it on the pairs $(\sigma_j,\sigma_l)$ for $j<l$ as follows, and extending it in a skew-symmetric bilinear way.
\begin{eqnarray}\label{alpha}
\alpha_i(\sigma_j,\sigma_l)=\left\{\begin{array}{ll}
\zeta_p& \text{if  } 1\leq j\leq i, \text{and } l=j+i,\\
1 &\text{otherwise.}
\end{array}\right.\end{eqnarray}
It is easily checked that the radical Rad$(\alpha_i)=\{g\in G|\ \ \alpha_i(g,h)=1,\forall h\in G\}$ is generated by $s-2i$ elements
\begin{equation}\label{radalpha}
\text{Rad}(\alpha_i)=\langle \sigma_j\rangle_{j=2i+1}^s.
\end{equation}
\begin{theorem}\label{complete}
With the above notation, $\{\alpha_i\}_{i=0}^{\lfloor\frac{s}{2}\rfloor}$ is a complete set of representatives for the Hamming isometry classes of complex twisted $(C_p)^s$-algebras.
\end{theorem}
\begin{proof}
The theorem follows from a standard result about skew-symmetric forms on vector spaces (see, e.g. \cite[Theorem 4]{Albert38}),
using the above identification of isometry classes of twisted $(C_p)^s$-algebras with congruence classes of skew-symmetric forms of $s$-dimensional spaces over $\F_p$.
\end{proof}
As for semisimplicity, again by \cite[Theorem 3.3]{AGdR}, for every prime $p$, positive integer $s$, and cocycle $f\in Z^2((C_p)^s,\C^*)$, the twisted group algebra $\C^f*(C_p)^s$ is semisimple.
Its Artin-Wedderburn decomposition as a direct sum of complex matrix algebras (of the same dimension) is given in terms of its corresponding bicharacter $\alpha_f$ \eqref{ssf}
with a representative $\alpha_i$ \eqref{alpha}.
\begin{theorem}\label{ir}
Let $f\in Z^2((C_p)^s,\C^*)$, such that its corresponding bicharacter is congruent to $\alpha_i$ for some $0\leq i\leq\lfloor\frac{s}{2}\rfloor$. Then the twisted group algebra
$\C^f*(C_p)^s$ is isomorphic to a direct sum of $p^{s-2i}$ copies of $p^i\times p^i$ complex matrix algebras, namely,
$$\C^f*(C_p)^s\cong \left(M_{p^i}(\C)\right)^{p^{s-2i}}.$$
In particular, two twisted group algebras $\C^{f_1}*(C_p)^s$ and $\C^{f_2}*(C_p)^s$ are isometric if and only if they are isomorphic (compare with Remark \ref{isomorphic}).
\end{theorem}
\begin{proof}

By \cite{Y} (see also \cite[Corollary 8.2.10]{karp3}), $\C^f*(C_p)^s$ is isomorphic to a direct sum of copies of complex matrix algebras of the same dimension. That is
\begin{equation}\label{nl}
\C^f*(C_p)^s\cong \left(M_{n}(\C)\right)^{l}
\end{equation}
for some $n,l$ which, by comparing dimensions, satisfy
\begin{equation}\label{nlpr}
n^2\cdot l=p^s.
\end{equation}
The number $l$ of such copies is the dimension of the center of the semisimple algebra $\C^f*(C_p)^s$, which in turn is the order of Rad$(\alpha_f)$. Hence,
\begin{equation}\label{lRad}
l=|\text{Rad}(\alpha_f)|=|\text{Rad}(\alpha_i)|,
\end{equation}
where the rightmost equality in \eqref{lRad} follows from the fact that
radicals of congruent forms are of the same cardinality. By \eqref{radalpha} the rank of Rad$(\alpha_i)<G$ is $s-2i$, therefore
\begin{equation}\label{radcard}
|\text{Rad}(\alpha_i)|=p^{s-2i}.
\end{equation}
By \eqref{nlpr},\eqref{lRad} and \eqref{radcard} we obtain $l=p^{s-2i}$ and $n=p^{i}$. Plugging these parameters into \eqref{nl} completes the proof.
\end{proof}
\begin{remark}
With the notation of Theorem \ref{ir}, the crossed product ambient space $\C^f*(C_p)^s$ admits $p^{s-2i}$ irreducible codes (out of $2^{p^{s-2i}}$ codes all in all).
Although twisted group algebras $\C^f*(C_p)^s$ are not classical crossed products, they may be simple as can be verified by taking $s$ to be any even number and putting $i=\frac{s}{2}$ in
Theorem \ref{ir}. The corresponding $f\in Z^2((C_p)^s,\C^*)$ is called {\it non-degenerate} (see, e.g. \cite[\S 3.1]{gs} for a more general definition), and the only $\C^f*(C_p)^s$-codes
are the zero space and the entire ambient space.
\end{remark}
We end this section with a short discussion without proofs about Hamming isometry of twisted group algebras $(\F_{q^r})^f*(C_p)^s$ of such elementary abelian groups $(C_p)^s$
over finite fields $\F_{q^r}$.
Firstly, these algebras are semisimple if and only if $p\neq q$ \cite[Theorem 3.3]{AGdR}. Next, the corresponding cohomology is
$$H^2((C_p)^s,\F_{q^r}^*)=\left\{\begin{array}{ll}
1& \text{if  gcd}(p,q^r-1)=1,\\
C_p^{{s+1}\choose{2}} & \text{if }p\text{ divides }q^r-1.
\end{array}\right.$$

When gcd$(p,q^r-1)=1$ (including the modular case $p=q$), then it is obvious that there is only one Hamming isometry class, namely of the ordinary group algebra $\F_{q^r}(C_p)^s$.
The complementary case is given as follows.
\begin{theorem}
Let $p\in\Delta(q^r-1)$.
Then there is a one-to-one correspondence between the Hamming isometry classes of the twisted group algebras
$(\F_{q^r})^f*(C_p)^s$ of an elementary abelian group $(C_p)^s$ over a finite field $\F_{q^r}$, and the $\lfloor\frac{3s}{2}\rfloor+1$ isomorphism types of groups $P$ which admit a central extension
$$1\to C_p\to P\to (C_p)^s\to 1.$$
\end{theorem}

\noindent{\bf Acknowledgement.} We thank O. Schnabel for his useful suggestions.

\end{document}